\documentclass{amsart}

\usepackage{comment}

\usepackage[mathscr]{eucal}
\usepackage{amssymb}
\usepackage{graphicx}
\usepackage{psfrag} 
\usepackage{epstopdf} 
\usepackage[usenames,dvipsnames]{color}
\usepackage[normalem]{ulem}
\usepackage{amsthm}
\usepackage{bbm}
\usepackage{enumerate}
\usepackage{array}
\usepackage{amsmath}
\usepackage{leftindex}

\usepackage{hyperref}
\usepackage[T1]{fontenc}

\hypersetup{colorlinks=true,linkcolor={Brown},citecolor={Brown},urlcolor={Brown}}

\usepackage{cleveref}

\numberwithin{equation}{section}
	\makeatletter
	\def\l@subsection{\@tocline{2}{0pt}{2.5pc}{5pc}{}}

\makeindex

\usepackage[all]{xy}
\CompileMatrices

\newdir{ >}{{}*!/-10pt/\dir{>}}

\swapnumbers 

\newtheorem{Thm}[equation]{Theorem}
\newtheorem{ThmIntro}{Theorem}
\newtheorem*{Thm*}{Theorem}
\newtheorem{Prop}[equation]{Proposition}
\newtheorem{Lem}[equation]{Lemma}
\theoremstyle{remark}
\newtheorem{Def}[equation]{Definition}

\newtheorem{Not}[equation]{Notation}
\newtheorem{Exa}[equation]{Example}

\newtheorem*{Conv*}{Conventions}

\newtheorem{Rec}[equation]{Recollection}

\newtheorem{Rem}[equation]{Remark}
\newtheorem{Rems}[equation]{Remarks}

\theoremstyle{definition}
\newtheorem*{Ack*}{Acknowledgements}
\newtheorem*{Cont*}{Contents}

\newcommand{\nc}{\newcommand}
\nc{\dmo}{\DeclareMathOperator}

\dmo{\Ab}{\mathsf{Ab}}
\dmo{\AbMon}{AbMon}
\dmo{\Abelem}{Abelem}
\dmo{\Aut}{Aut}
\dmo{\Bi}{bi}
\dmo{\Bisets}{Bisets}
\dmo{\Br}{Br}
\dmo{\DER}{\mathsf{DER}}
\dmo{\ADDER}{\mathsf{ADDER}}
\dmo{\coev}{coev}
\dmo{\Coloc}{Coloc}
\dmo{\ev}{ev}
\dmo{\Fib}{Fib}
\dmo{\Free}{Free}
\dmo{\Id}{Id}
\dmo{\loc}{Loc}
\dmo{\rmI}{I}
\dmo{\rmL}{L}
\dmo{\rmR}{R}
\dmo{\Spc}{Spc}
\dmo{\thick}{Thick}
\dmo{\chara}{char}%
\dmo{\coh}{coh} 
\dmo{\Coind}{CoInd}
\dmo{\coker}{coker}

\dmo{\cone}{Cone}
\dmo{\Cone}{Cone}
\dmo{\Der}{D}
\dmo{\Ch}{Ch}
\nc{\Rder}{\mathrm{R}} 
\nc{\Lder}{\mathrm{L}} 
\dmo{\Khocat}{K}
\dmo{\End}{End}
\dmo{\Ext}{Ext}
\dmo{\rmH}{H}
\dmo{\Ho}{Ho}
\dmo{\Hom}{Hom}
\dmo{\id}{id}
\dmo{\Img}{Im}
\dmo{\incl}{incl}
\dmo{\Ind}{Ind}
\dmo{\ind}{ind}
\dmo{\Indk}{Ind_{\kappa}}
\dmo{\CoInd}{CoInd}

\dmo{\PSh}{PSh}
\dmo{\Top}{Top}

\dmo{\Ker}{Ker}
\dmo{\Les}{Les}
\dmo{\Map}{Map}%
\dmo{\Mod}{\mathsf{Mod}}
\dmo{\GrMod}{GrMod}
\dmo{\lax}{lax}
\dmo{\modname}{mod}%
\dmo{\grmod}{grmod}
\dmo{\Mor}{Mor}%
\dmo{\Obj}{Obj}
\dmo{\Or}{Or}
\dmo{\Oldorbit}{\mathcal{O}} 
\dmo{\Fix}{Fix} 
\dmo{\Ev}{Ev} 
\dmo{\pr}{pr}
\dmo{\canin}{in} 
\dmo{\Proj}{Proj} 
\dmo{\Inj}{Inj} 
\dmo{\proj}{proj}
\dmo{\Qcoh}{Qcoh}
\dmo{\rank}{rank}
\dmo{\Res}{Res}
\dmo{\res}{res}
\dmo{\Defl}{Def}
\dmo{\Infl}{Inf}
\dmo{\Iso}{Iso}
\dmo{\Rname}{R}
\dmo{\Sp}{Sp} 
\dmo{\SH}{\mathsf{SH}}
\nc{\SHp}{\SH_{(p)}}
\dmo{\smallb}{b}
\dmo{\smallperf}{perf}
\dmo{\Spec}{Spec}
\dmo{\Spech}{Spec^h}
\dmo{\Stab}{Stab}
\dmo{\stab}{stab}
\dmo{\supp}{Supp}
\dmo{\switch}{switch}
\dmo{\TTR}{TTR}
\dmo{\Spanname}{{\sf Span}}
\dmo{\map}{map}
\dmo{\Rel}{Rel}

\nc{\Ivo}[1]{{\color{OliveGreen}#1}}
\nc{\Paul}[1]{{\color{Blue}#1}}
\nc{\Pout}[1]{\Paul{\sout{#1}}}
\nc{\Iout}[1]{\Ivo{\sout{#1}}}
\nc{\Prule}{\Paul{\smallbreak\hrule\smallbreak}}

\nc{\SEcell}{\rotatebox[origin=c]{45}{$\Downarrow$}} 
\nc{\NEcell}{\rotatebox[origin=c]{135}{$\Downarrow$}} 
\nc{\SWcell}{\rotatebox[origin=c]{-45}{$\Downarrow$}} 
\nc{\NWcell}{\rotatebox[origin=c]{-135}{$\Downarrow$}} 
\nc{\Scell}{\rotatebox[origin=c]{0}{$\Downarrow$}} 
\nc{\Ncell}{\rotatebox[origin=c]{0}{$\Uparrow$}} 
\nc{\Ecell}{\Rightarrow}
\nc{\oEcell}[1]{\overset{\scriptstyle #1}{\Ecell}}
\nc{\isoEcell}{\overset{\sim}{\,\Ecell\,}}
\nc{\isocell}[1]{\undersett{ #1}\isoEcell}
\nc{\Isocell}[1]{\undersett{ #1}{\overset{\sim}{\Longrightarrow}}}
\nc{\Wcell}{\rotatebox[origin=c]{90}{$\Uparrow$}} 
\nc{\Span}{\Spanname}
\nc{\Spanhat}{\textrm{\sf S}\widehat{\textrm{\sf pan}}} 
\nc{\tSpan}{\pih{\Spanname}}
\nc{\IFF}{$\Leftrightarrow$}
\nc{\ass}{\mathrm{ass}} 
\nc{\lun}{\mathrm{lun}} 
\nc{\run}{\mathrm{run}} 
\nc{\fun}{\mathrm{fun}} 
\nc{\un}{\mathrm{un}} 
\nc{\Crich}{\underline{\cat{C}}}
\nc{\uA}{\underline{A}}
\nc{\doublequot}[3]{#1\backslash #2/#3}
\nc{\HGK}{\doublequot HGK}
\nc{\quadtext}[1]{\quad\textrm{#1}\quad}
\nc{\qquadtext}[1]{\qquad\textrm{#1}\qquad}
\nc{\PZG}{\cat{C}_{\bbZ}(\bbZ G)}
\nc{\TTRK}{\TTR(\cat K)}
\nc{\psets}{\mathsf{-sets}_\sbull}
\nc{\Gsets}{G\mathsf{-sets}}
\nc{\Hsets}{H\mathsf{-sets}}
\nc{\AddK}{\Add^{\Sigma}(\cat K)}
\nc{\adj}{\dashv\,}
\nc{\adjto}{\rightleftarrows}
\nc{\AK}{A\MModcat{K}}
\nc{\BK}{B\MModcat{K}}
\nc{\bbA}{\mathbb{A}}
\nc{\bbB}{\mathbb{B}}
\nc{\bbC}{\mathbb{C}}
\nc{\bbD}{\mathbb{D}}
\nc{\bbF}{\mathbb{F}}
\nc{\bbI}{\mathbb{I}}
\nc{\bbM}{\mathbb{M}}
\nc{\bbN}{\mathbb{N}}
\nc{\bbP}{\mathbb{P}}
\nc{\bbQ}{\mathbb{Q}}
\nc{\bbR}{\mathbb{R}}
\nc{\bbZ}{\mathbb{Z}}
\nc{\bbZp}{\mathbb{Z}_{(p)}}
\nc{\Sphere}{\mathbb{S}} 
\nc{\cat}[1]{\mathcal{#1}}
\nc{\Displ}{\displaystyle}
\nc{\ie}{{\sl i.e.}\ }
\nc{\cf}{{\sl cf.}\ }
\nc{\into}{\mathop{\rightarrowtail}}
\nc{\inv}{^{-1}}
\nc{\isoto}{\buildrel \sim\over\to}
\nc{\isotoo}{\mathop{\buildrel \sim\over\too}}
\nc{\onto}{\mathop{\twoheadrightarrow}}
\nc{\too}{\mathop{\longrightarrow}\limits}
\nc{\xytriangle}[7]{\xymatrix@C=#7em{#1\ar[r]^-{\Displ #4} & #2 \ar[r]^-{\Displ #5}&#3\ar[r]^-{\Displ #6}&T #1}}
\nc{\ababs}{{\sl ab absurdo}}
\nc{\adh}[1]{\overline{#1}}
\nc{\adhoc}{{\sl ad hoc}}
\nc{\adhpt}[1]{\adh{\{#1\}}}
\nc{\afortiori}{{\sl a fortiori}}
\nc{\aka}{{a.\,k.\,a.}\ }
\nc{\ala}{{\sl \`a la}\ }
\nc{\apriori}{{\sl a priori}}
\nc{\Autcat}[1]{\Aut_{\cat #1}}
\nc{\cO}{\mathcal{O}}
\nc{\calO}{\mathcal{O}}
\nc{\eg}{{\sl e.g.}}
\nc{\eps}{\varepsilon}
\nc{\equalby}[1]{\overset{\textrm{#1}}{=}}
\nc{\gm}{\mathfrak{m}}
\nc{\Homcat}[1]{\Hom_{\cat #1}}
\nc{\Morcat}[1]{\Mor_{\cat #1}}
\nc{\hook}{\hookrightarrow}

\nc{\Idcat}[1]{\Id_{\cat{#1}}}
\nc{\ideal}[1]{\langle #1\rangle}
\nc{\ihom}{{\mathsf{hom}}} 
\nc{\ihomcat}[1]{\ihom_{\cat #1}}
\nc{\Kcat}[1]{#1\MModcat{K}}
\nc{\KP}{\cat{K}_{\cat P}}
\nc{\loccit}{{\sl loc.\ cit.}}
\nc{\lind}{\rmL\!}
\nc{\RR}{\rmR\!}
\nc{\Lotimes}{\otimes^{\rmL}}
\nc{\Mid}{\,\bigm|\,}
\nc{\MMod}{\,\textsf{-}\Mod}%
\nc{\Exact}{\mathfrak K\,\text{-}\mathrm{Exa}^\mathbb Z/2_\infty} 
\nc{\MModcat}[1]{\MMod_{\cat #1}}%
\nc{\mmod}{\,\text{--}\modname}%
\nc{\mmodb}{\mmod^\sbull}%
\nc{\op}{{\mathrm{op}}}
\nc{\co}{{\mathrm{co}}}
\nc{\costar}{**}
\nc{\oto}[1]{\overset{#1}\to}
\nc{\ointo}[1]{\overset{#1}{\rightarrowtail}}
\nc{\loto}[1]{\overset{#1}{\leftarrow}}
\nc{\otoo}[1]{\overset{#1}{\,\longrightarrow\,}}
\nc{\lotoo}[1]{\overset{#1}{\,\longleftarrow\,}}
\nc{\ourfrac}[2]{\genfrac{}{}{0pt}{}{\Displ #1}{\scriptstyle #2}}
\nc{\ouriff}{\Leftrightarrow}
\nc{\oursetminus}{\!\smallsetminus\!}
\nc{\potimes}[1]{^{\otimes #1}}
\nc{\pproj}{\,\text{-}\proj}
\nc{\ptimes}[1]{^{\times #1}}
\nc{\dd}[1]{_{{\scriptscriptstyle(#1)}}}
\nc{\uu}[1]{^{{\scriptscriptstyle(#1)}}}
\nc{\pushout}{\textrm{\rm p.o.}}
\nc{\qp}{q_{_{\scriptstyle \cat P}}\!}%
\nc{\Rcat}[1]{\Rname_{\cat #1}^\sbull}
\nc{\rdto}{}
\nc{\restr}[1]{{|_{\scriptstyle #1}}}
\nc{\RK}{\Rcat{K}}
\nc{\sbull}{{\scriptscriptstyle\bullet}}
\nc{\SET}[2]{\bigl\{\,#1\Mid#2\,\bigr\}}
\nc{\SHA}{\SH{}^{\bbA^{1}}}
\nc{\SHfin}{\SH^{\text{\rm fin}}}
\nc{\smat}[1]{\left(\begin{smallmatrix} #1 \end{smallmatrix}\right)}
\nc{\SpcAK}{\Spc(A\MModcat{K})}
\nc{\SpcK}{\Spc(\cat K)}
\nc{\suppcat}[1]{\supp(\cat #1)}
\nc{\then}{\Rightarrow}
\nc{\tideal}[1]{\ideal{#1}}
\nc{\unit}{\mathbbm{1}}
\nc{\unitcat}[1]{\unit_{\cat #1}}
\nc{\vcorrect}[1]{{\vphantom{\vbox to #1em{}}}}
\nc{\onept}{\mathrm{B}} 
\nc{\undersett}[1]{\underset{\scriptstyle #1}}
\nc\noloc{\nobreak\mspace{6mu plus 1mu}{:}\nonscript\mkern-\thinmuskip\mathpunct{}\mspace{2mu}}

\nc{\HG}{\!{}^{^H}\overline{G}}
\nc{\uY}{\widetilde{Y}}

\nc{\ADD}{\mathsf{ADD}}
\nc{\MONADD}{\mathsf{MONADD}}
\nc{\SMONADD}{\mathsf{SMONADD}}
\nc{\Add}{\mathsf{Add}}
\nc{\SAD}{\mathsf{SAD}}
\nc{\Sad}{\mathsf{Sad}}
\nc{\Cat}{\mathsf{Cat}}
\nc{\CCat}{\textsf{-}\mathsf{Cat}}
\nc{\CAT}{\mathsf{CAT}}
\nc{\Dk}{\dual_{\kappa}}
\nc{\Dkk}{\dual_{\kappa'}}
\nc{\bs}{\backslash}
\nc{\biCpt}{\mathrm{biCpt}}
\nc{\biLCpt}{\mathrm{biLCpt}} 
\nc{\Groupoid}{\mathsf{Groupoid}}
\nc{\groupoid}{\mathsf{gpd}}
\nc{\gpd}{\groupoid}
\nc{\faithful}{\mathsf{faithful}}
\nc{\faith}{\mathsf{faithf}}
\nc{\exact}{\mathsf{ex}}
\nc{\smallfaithful}{\mathsf{f}}
\nc{\smallfused}{\mathsf{fus}}
\nc{\groupoidf}{\groupoid{}^{\smallfaithful}}
\nc{\groconn}{\groupoid_{\mathsf{conn}}}
\nc{\gps}{\mathsf{groups}} 
\nc{\group}{\mathsf{group}} 
\nc{\groupshort}{\mathsf{gr}}
\nc{\gpdG}{{\groupoidf_{\!\smallslash\!G}}} 
\nc{\GinG}{{\groupoidf_{G}}}
\nc{\gpdGfuz}{{\groupoid^{\smallfused}_{\!\smallslash\!G}}} 
\nc{\spanG}{{\widehat{\mathsf{gp}\,\,}\!\!\mathsf{d}}{}^\smallfaithful_{\!{}^{\scriptscriptstyle/}\!G}}
\nc{\biset}{\mathsf{biset}} 
\nc{\rfree}{\mathsf{rf}} 
\nc{\bifree}{\mathsf{bif}} 
\nc{\conj}{\mathsf{conj}} 
\nc{\smallslash}{{}^{\scriptscriptstyle/}}
\nc{\smallbs}{{}^{\scriptscriptstyle\backslash}}
\nc{\doublebs}{\smallbs\!\smallbs}

\nc{\Set}{\mathsf{Set}}
\nc{\set}{\mathsf{set}} 
\nc{\sset}{\textrm{-}\set}
\nc{\ssetfused}{\textrm{-}\underline{\set}} 
\nc{\ssetfuz}{\sset^{\smallfused}} 
\nc{\Comp}{\mathsf{Top}^{\mathsf{comp}}}
\nc{\pih}[1]{\tau_{1}#1}
\nc{\all}{\mathsf{all}}

\dmo{\Fun}{\mathrm{Fun}} 
\dmo{\PsFun}{\mathsf{PsFun}} 
\dmo{\PsFunlax}{\mathsf{PsFun}_{\mathsf{lax}}}
\dmo{\PsFunoplax}{\mathsf{PsFun}_{\mathsf{oplax}}}
\dmo{\BCDex}{\mathsf{BCDex}_{\II\mathsf{-str}}}
\dmo{\BCDexdex}{\mathsf{BCDex}_{\II\mathsf{-dex}}}
\dmo{\biMack}{\mathsf{Mack}} 
\dmo{\Mackey}{\mathsf{Mack}} 
\nc{\MMackey}{\,\textsf{-}\Mackey} 
\dmo{\twoFun}{2\mathsf{Fun}}

\nc{\Muniv}{\cat{M}^{\mathsf{univ}}}

\nc{\lG}{{}_{{\color{Gray}\scriptscriptstyle G}}}
\nc{\lH}{{}_{{\color{Gray}\scriptscriptstyle H}}}
\nc{\rG}{_{{\color{Gray}\!\scriptscriptstyle G}}}
\nc{\rH}{_{{\color{Gray}\!\scriptscriptstyle H}}}
\nc{\rK}{_{{\color{Gray}\!\scriptscriptstyle K}}}

\nc{\dual}{\Delta}

\nc{\ra}{\rightarrow}
\nc{\xra}{\xrightarrow}
\nc{\lto}{\leftarrow}
\nc{\olto}[1]{\overset{#1}\lto}

\nc{\C}{\mathbb{C}} 
\nc{\Cont}{\mathrm{C}} 
\nc{\Rep}{\mathrm{R}} 
\nc{\KK}{\mathsf{KK}} 
\nc{\Calg}{\mathsf{C^*\text{-}Alg}} 
\nc{\Kth}{\mathrm{K}} 
\nc{\Cell}[1]{\mathsf{Cell}(#1)}
\nc{\bigCell}[1]{\mathsf{Cell}(#1)^\mathrm{big}}
\nc{\bigC}{{\cat{C}^\mathrm{big}}} 
\nc{\bigT}{{\cat{T}_\mathrm{big}}} 
\nc{\Modules}{\mathsf{Mod}}
\nc{\Alg}{\mathsf{Alg}}
\nc{\Sep}{\mathsf{Sep}}
\nc{\BurnG}{\cat{A}(G)}
\nc{\Loc}[1]{\loc\langle\, #1 \,\rangle}
\nc{\Thick}[1]{\thick\langle\, #1 \,\rangle}
\nc{\tensLoc}[1]{\loc_\otimes\!\langle\, #1 \,\rangle}
\nc{\tensThick}[1]{\thick_\otimes\!\langle\, #1 \,\rangle}
\nc{\Cstar}[1]{\textrm{C*alg}^{#1}}
\nc{\Cstarsep}[1]{\textrm{C*alg}_{\textrm{sep}}^{#1}}

\nc{\inftyKK}[1]{\mathbf{KK}^{#1}} 
\nc{\inftyKKloc}[1]{\mathbf{KK}^{#1}_{\oplus}} 
\nc{\inftyKKsep}[1]{\mathbf{KK}^{#1}_{\mathsf{sep}}} 
\nc{\inftyInd}[1]{\mathbf{Ind}(#1)}
\nc{\inftyCell}[1]{\mathbf{Cell}(#1)}
\nc{\inftyCellsep}[1]{\mathbf{Cell}(#1)_{\mathsf{sep}}}
\nc{\inftyFun}{\mathrm{Fun}}
\nc{\kk}[1]{\mathrm{kk}^{#1}}
\nc{\kksep}[1]{\mathrm{kk}_{\mathrm{sep}}^{#1}}

\nc{\cD}{\cat{D}}
\nc{\cG}{\cat{G}}
\nc{\cM}{\cat{M}}
\nc{\cN}{\cat{N}}

\nc{\DD}{\cat{D}}
\nc{\MM}{\cat{M}}
\nc{\NN}{\cat{N}}
\nc{\GG}{\mathbb{G}}

\nc{\gammap}[1]{\gamma^{(#1)}}
\nc{\what}[1]{\widehat{\cat{#1}}}

\nc{\und}[1]{{\kern1pt\underline{\kern-1pt{#1}\kern-1.5pt}\kern1.5pt}}

\nc{\Funplus}{\Fun_{+}}
\nc{\Mack}[1]{(Mack\,\ref{Mack-#1})}

\begin{document}


\title{A general Greenlees--May splitting principle}
\author{Serge Bouc}
\author{Ivo Dell'Ambrogio}
\author{Rub\'en Martos}
\date{\today}

\address{\ \smallbreak
\noindent SB: Univ.\,de Picardie - Jules Verne, LAMFA UMR 7352, F-80039 Amiens, France
}
\email{serge.bouc@u-picardie.fr}
\urladdr{https://www.lamfa.u-picardie.fr/bouc/}

\address{\ \smallbreak
\noindent ID: Univ.\,Lille, CNRS, UMR 8524 - Laboratoire Paul Painlev\'e, F-59000 Lille, France
}
\email{ivo.dell-ambrogio@univ-lille.fr}
\urladdr{https://idellambrogio.github.io/}

\address{\ \smallbreak
\noindent RM: Univ.\,Lille, CNRS, UMR 8524 \!-\! Laboratoire Paul Painlev\'e, F-59000 Lille, France
}
\email{ruben.martos@univ-lille.fr}
\urladdr{https://sites.google.com/view/ruben-martos/}

\begin{abstract}
In equivariant topology, Greenlees and May used Mackey functors to show that, rationally, the stable homotopy category of $G$-spectra over a finite group~$G$ splits as a product of simpler module categories. 
We extend the algebraic part (also independently proved by Th\'evenaz and Webb) of this classical result to Mackey modules over an arbitrary Green functor, and use the case of the complex representation ring Green functor to obtain an algebraic model of the rational equivariant Kasparov category of $G$-cell algebras. 
\end{abstract}


\thanks{The second and third authors are supported by the Labex CEMPI (ANR-11-LABX-0007-01)}

\maketitle


\section{Introduction and results}
\label{sec:intro}%
\medskip

Let $G$ be a finite group, which we fix throughout the article.

In equivariant stable homotopy theory, a well-known result of Greenlees--May \cite[App.\,A]{GreenleesMayTate95} states that the stable homotopy category of rational $G$-spectra splits as a product of module categories
\begin{equation} \label{eq:intro-equiv}
\SH(G)_\mathbb Q \simeq \underset{\text{Cl}(H),\ H\leq G}{\prod}  \mathbb Q [W_G(H)]\text{-}\Modules^\mathbb Z ,
\end{equation}
one for each $G$-conjugacy class $\mathrm{Cl}(H)$ of subgroups $H\leq G$. Here $W_G(H)= N_G(H)/H$ denotes the Weyl group of $H$ in~$G$, and the corresponding factor is the semisimple abelian category of $\mathbb Z$-graded modules over the rational group algebra $\mathbb Q [W_G(H)]$ (the decoration $(\ldots)^\mathbb Z$ indicates graded objects).
This provides a canonical decomposition and an explicit algebraic model for rational equivariant cohomology theories and is thus of fundamental importance in equivariant topology. 

The equivalence is induced\footnote{The present description of the equivalence is not the one given in \cite{GreenleesMayTate95} in terms of idempotents elements of the Burnside ring, but can be easily shown to be equivalent; see \Cref{Rem:idempotents}.} by the composite functor 
\begin{equation} \label{eq:intro-composite}
\xymatrix{
\SH(G)
 \ar[r]^-{\underline{\pi}_*} & 
\Mackey(G)^\mathbb Z
 \ar[r]^-{\Br} &
{\underset{\text{Cl}(H),\ H\leq G}{\prod}  \mathbb Z [W_G(H)]\text{-}\Modules^\mathbb Z}
}
\end{equation}
which is already defined integrally. 
First one sends a $G$-spectrum $X$ to its equivariant homotopy groups, which form a ($\mathbb Z$-graded, abelian groups-valued) $G$-Mackey functor $\underline{\pi}_*(X)\in \Mackey(G)^\mathbb Z$.
Then one can form the so-called \emph{Brauer quotient} of a Mackey functor at every~$H$, which collectively form a functor~$\Br$ as above.

To see that \eqref{eq:intro-composite} induces an equivalence rationally, the harder part is to prove the same is true for~$\Br$.
Note that the latter is a purely algebraic statement which already holds when inverting the order of~$G$ and before taking graded objects; it is essentially already contained in \cite{Greenlees92} as well as in the independent work of Th\'evenaz--Webb~\cite{ThevenazWebb95}.
Once the latter is established, it immediately follows that rational Mackey functors form a semisimple category, and then some fairly standard homological algebra implies that $\underline{\pi}_*$ is also a rational equivalence.

\begin{center} $***$ \end{center}

The main result of this article is a generalization of the algebraic side of the story to general Green functors \cite{Bouc97}. 
If $R$ is any Green functor for~$G$, one observes (\Cref{Rems:Psi}) that the construction of Brauer quotients defines a functor
	\begin{equation} \label{eq:intro-new-Brauer-functor}
	\Br^{R} \colon R\MMackey(G) \longrightarrow \underset{\mathrm{Cl}(H),\ H\leq G}{\prod} \overline{R}(H) \rtimes_c  W_G(H)\MMod
	\end{equation}
from the category of Mackey $R$-modules (\ie Mackey functors equipped with a left $R$-action) to a product of module categories over certain skew group rings $\overline{R}(H)\rtimes_cW_G(H)$. Here $\overline{R}(H)$ is the Brauer quotient of $R$ at~$H$, and the skew group ring incorporates the residual conjugation action of $W_G(H)$ on it.
We prove:
\begin{ThmIntro}[{\cf \Cref{thm.RightAdjointMackMod}}]
\label{Thm:intro-splitting}
Let $G$ be a finite group, $R$ a Green functor for~$G$, and $Q\subseteq \mathbb Q$ a subring of the rationals containing~$|G|^{-1}$. 
Denote by $R_Q$ the Green functor obtained by tensoring $R$ with~$Q$.
Then \eqref{eq:intro-new-Brauer-functor} induces an equivalence 
\[
R_Q \MMackey(G) \simeq \underset{\mathrm{Cl}(H),\ H\leq G}{\prod} Q\otimes_\mathbb Z \overline{R}(H) \rtimes_c  W_G(H)\MMod.
\]
\end{ThmIntro}
By taking $R$ to be the Burnside ring Green functor, one recovers as a special case the Greenlees--May\,/ Th\'evenaz--Webb rational splitting for plain Mackey functors.

\begin{center} $***$ \end{center}

It is known that the Greenlees-May splitting \eqref{eq:intro-equiv} is an equivalence of tensor categories, with respect to the smash product of rational $G$-spectra on the left-hand side and the usual product of group representations on the right-hand side; see \eg\ \cite{BarnesKedziorek22}
(it can also be deduced from a zig-zag of symmetric monoidal Quillen equivalences \cite{Barnes09} \cite{Kedziorek17}).
On the algebraic side, this is essentially contained in~\cite[\S8-10]{Bouc97}; \cf \Cref{sec:tensor}.

More generally, the category $R\MMackey(G)$ becomes symmetric monoidal for any commutative Green functor~$R$. 
Moreover, each Brauer quotient $\overline{R}(H)$ is a commutative ring, hence we may tensor two $\overline{R}(H)\rtimes_cW_G(H)$-modules over it and equip the result with the diagonal group action. 
We establish that the Brauer quotient is compatible with these tensor structures in full generality:

\begin{ThmIntro}[{\cf \Cref{Thm:Brauer_tensor}}]
\label{Thm:intro-tensor}
For any commutative Green functor $R$ for~$G$, the functor \eqref{eq:intro-new-Brauer-functor} naturally inherits a strong symmetric monoidal structure.
\end{ThmIntro}

\begin{center} $***$ \end{center}

Finally, we illustrate the usefulness of the above results by applying them to the $G$-equivariant Kasparov category $\KK^G$ of separable complex $G$-C*-algebras \cite{Kasparov88}. 
More precisely we consider  (analogously to the topological situation) the full tensor-triang\-ulat\-ed subcategory $\Cell{G}$ of \emph{$G$-cell algebras}, \ie the localizing subcategory generated by the complex function algebras $\Cont(G/H)$, as well as its rationalization $\Cell{G}_\mathbb Q$. We obtain the following splitting theorem:

\begin{ThmIntro} [{\cf \Cref{sec:Module-Mackey-Gcell}}]
\label{Thm:rational_cell_vs_Mack}
For any finite group $G$, taking equivariant K-theory and Brauer quotients induces an equivalence
		\[ 
		\Cell{G}_{\mathbb{Q}}\simeq {\prod_{\mathrm{Cl}(H),\, H\leq G\; \textrm{cyclic}}}  \mathbb{Q}(\zeta_{|H|}) \rtimes_c W_G(H)\MMod^{\mathbb Z/2}_{\aleph_1} 
		\]
of symmetric monoidal semisimple abelian categories. 
For the right-hand side, we choose a cyclic subgroup $H$ in each $G$-conjugacy class and a generator $\langle g\rangle = H$; then the skew group ring at $H$ has coefficients in the cyclic field extension $\mathbb Q(\zeta_{|H|})$, on which $W_G(H)$ acts  by 
\[ 
c_w (\zeta_{|H|})= \zeta_{|H|}^{m_H(w)} \quad  (w\in W_G(H))
\] 
for the unique $m_H(w)\in (\mathbb Z/|H|)^\times$ such that $g^{m(w)}=w^{-1}gw$.
\end{ThmIntro}

Here we grade our modules over $\mathbb Z/2$, because of Bott periodicity, and the decoration $\aleph_1$ says that we only need consider countable modules (as $\KK^G$ only admits countable coproducts).  
Besides these details, the proof can be run analogously as for~\eqref{eq:intro-equiv}. 
Homotopy groups are replaced by equivariant topological K-theory of C*-algebras, which takes values in Mackey modules over the complex representation ring Green functor~$\Rep^G$ by results of \cite{DellAmbrogio14}. 
The result follows from Theorems~\ref{Thm:intro-splitting} and~\Cref{Thm:intro-tensor} for $R=\Rep^G$ and $Q=\mathbb Q$, together with the known calculation of the Brauer quotients of~$\Rep^G_\mathbb Q$.

In particular \Cref{Thm:rational_cell_vs_Mack} provides an explicit algebraic classification, up to rational equivariant KK-equivalence, of all separable $G$-C*-algebras which are $G$-cell algebras. We recall that the class of $G$-cell algebras is closed under several bootstrap operations and contains many of the separable $G$-C*-algebras of interest (though, for general~$G$, not all of them; \cf \cite[Rem.\,2.4]{DellAmbrogio14} \cite[\S3.1]{DEM14}).

\begin{Rem} \label{Rem:intro-coefficients}
Theorems~\ref{Thm:intro-splitting} and~\ref{Thm:intro-tensor}  can be easily generalized to Mackey and Green functors with coefficients in any commutative ring~$\mathbb K$, in which case the hypothesis reads ``suppose that $|G|$ is invertible in~$\mathbb K$'' and the resulting equivalence is of $\mathbb K$-linear (tensor) categories.
Note that if $R$ is defined over $\mathbb K$ and $|G|^{-1}\in \mathbb K$, then $R_Q = R$ and similarly $Q\otimes_\mathbb Z \overline{R}(H) \rtimes_c  W_G(H) =  \overline{R}(H) \rtimes_c  W_G(H)$ for all~$H$. (See also Remarks~\ref{Rem:scalar_ext} and~\ref{Rem:graded-cble-variants}.)
\end{Rem}

\begin{Rem}
Using the Universal Coefficients and K\"unneth spectral sequences established in~\cite{LewisMandell06}, our Theorems~\ref{Thm:intro-splitting} and~\ref{Thm:intro-tensor} (or rather their evident $\mathbb Z$-graded versions) can be used to generalize the Greenlees--May equivalence~\eqref{eq:intro-equiv} to the rational stable homotopy category of modules over a (commutative) $G$-equivariant $S$-algebra~$A$, whenever its rationalized homotopy groups Green functor $\mathbb Q\otimes_\mathbb Z(\underline{\pi}_*A)$ has semisimple Brauer quotients at all subgroups of~$G$.

In fact, we expect the ``topological'' half of the proof of~\eqref{eq:intro-equiv} to also admit an even more comprehensive generalization, in terms of the Green 2-functors of~\cite{DellAmbrogio22}. 
This would allow the systematic application of Theorems~\ref{Thm:intro-splitting} and \ref{Thm:intro-tensor} to other examples of ``equivariant mathematics'' beyond stable homotopy and KK-theory. 
However, such an axiomatic setup would be rather heavy for this purpose and is better reserved to a general study of equivariant homological algebra not limited to the rational case.
\end{Rem}

\begin{Cont*}
After brief recollections on  skew group rings (\Cref{sec:skew-group-rings}) and Mackey functors (\Cref{sec:Mackey}) to fix notation, we describe the functor $\Br^R$ and prove \Cref{Thm:intro-splitting} (\Cref{sec:GreenlessMay-Module-Mackey}), prove \Cref{Thm:intro-tensor} (\Cref{sec:tensor}), and prove \Cref{Thm:rational_cell_vs_Mack} (\Cref{sec:Module-Mackey-Gcell}).
\end{Cont*}

\section{Skew group rings}
\label{sec:skew-group-rings}%
\medskip

Suppose a finite group $W$ acts on a ring $S$ by ring automorphisms, say via the group homomorphism $\alpha\colon W\to \Aut(S)$, $w\mapsto \alpha_w$.
The \emph{skew group ring} for this action, denoted 
\[
S \rtimes_\alpha W,
\]
is defined to be the free $S$-module $\bigoplus_{w\in W} Sw$ on the set~$W$, equipped with the multiplication defined by additively extending the rule
\[
(sw) (s'w') :=(s \alpha_w(s')) (ww')
\]
in both variables.
Note that $S \rtimes_\alpha W$ is a unital and associative ring but \emph{not} an $S$-algebra, indeed its multiplication is not $S$-bilinear.

\begin{Rem} \label{Rem:scalar_ext}
If $\mathbb K$ is any commutative ring, there is a unique extension of $c$ to a $W$-action $\tilde c$ on the $\mathbb K$-algebra $\mathbb K\otimes_\mathbb ZS$, and there is an evident isomorphism 
\[
(\mathbb K\otimes_\mathbb Z S)\rtimes_{\tilde c} W \cong \mathbb K \otimes_\mathbb Z (S \rtimes_c W)
\]
of $\mathbb K$-algebras. Accordingly, we may safely omit parentheses when extending scalars in a group ring, \eg\ as in the statement of our main theorems with $\mathbb K = Q$.
\end{Rem}

\begin{Rem}
\label{Rem:skew-modules}
The abelian category
\(
S\rtimes_\alpha W \MMod
\)
of left modules over the ring $S\rtimes_\alpha W$ admits the following alternative description. 
An object $V=(V, \varphi)$ is determined by an $S$-module $V$ equipped with a $W$-action $\varphi$ on the underlying abelian group of~$V$, which we will write as a set of abelian group automorphisms
\[
\varphi_w \colon V \overset{\sim}{\to} V \quad (w\in W)
\] 
satisfying $\varphi_w\varphi_{w'} = \varphi_{ww'}$ and $\varphi_e = \id_V$.
Moreover, the $S$-action and the  $W$-action are compatible in the sense that
\begin{equation} \label{eq:covariance}
\varphi_w(s \cdot v) = \alpha_w(s) \cdot \varphi_w(v)
\end{equation}
for all $w\in W$, $s\in S$ and $v\in V$. 
These two actions are obtained by restricting the given $S\rtimes_\alpha W$-action along the inclusions $S\to S\rtimes_\alpha W$, $s \mapsto se$, and $W\to S\rtimes_\alpha W$, $w\mapsto 1_S w$. 
Conversely, two such actions satisfying the above compatibility extend to a unique left $S\rtimes_\alpha W$-action on the abelian group~$V$. A morphism $\xi \colon (V, \varphi) \to (V'\varphi')$ of left modules is the same as an $S$-linear map $V\to V'$ which is equivariant for the group actions: $\varphi'_w \circ \xi = \xi \circ \varphi_w$ for all $w\in W$.
\end{Rem}

\begin{Exa} \label{Exa:non-skew}
If the action $\alpha$ is trivial, $S\rtimes_\alpha W$ is the usual group algebra~$S[W]$.
\end{Exa}

The classical Maschke theorem for group algebras also holds for skew group rings:

\begin{Lem}[{\cite[Corollary 0.2.(1)]{Montgomery80}}]
\label{Lem:semisimple}
If $S$ is a semisimple Artinian ring (\eg\ a field) containing $|W|^{-1}$, then $S\rtimes_\alpha W$ is also a semisimple Artinian ring.
In particular, all objects in the abelian category $S\rtimes_\alpha W\MMod$ are projective.
\qed
\end{Lem}

\begin{Rem} \label{Rem:tensor-skew}
When the ring $S$ is commutative, $S\rtimes_\alpha W\MMod$ admits a symmetric monoidal structure where $(V,\varphi)\otimes(V',\varphi'):= (V\otimes_S V', \varphi\otimes \varphi')$ for the diagonal action $(\varphi\otimes \varphi' )_w = \varphi_w \otimes \varphi_w$ ($w\in W$), and where 
$f \otimes f' := f\otimes_S f'$ on morphisms. The associative, unit and symmetry isomorphisms are those of the underlying $S$-modules.
The tensor unit object $\unit$ is provided by $(S,\alpha)$. 
\end{Rem}

\section{Mackey modules over a Green functor}
\label{sec:Mackey}%
\medskip

A \emph{Mackey functor} $M$ for the finite group~$G$ consists of an abelian group $M(H)$ for every subgroup $H\leq G$ together with three families of group homomorphisms
\begin{align*}
\res^H_K\colon M(H)\to M(K),
\quad \ind^H_K \colon M(K)\to M(H),
\quad c_{g,H} \colon M(H) \to M({}^gH)
\end{align*}
for all subgroups $K\leq H \leq G$ and elements $g\in G$, 
 respectively called \emph{restriction}, \emph{induction} and \emph{conjugation} maps; these are subject to certain relations.
 (We use the standard notations ${}^gH=gHg^{-1}$ and $H^g = g^{-1}Hg$ for the conjugates of a subgroup $H\leq G$ by an element $g\in G$.) 
Mackey functors (for~$G$) form an abelian category $\Mackey:= \Mackey_\mathbb Z(G)$, where a morphism $f \colon M\to M'$ is a family of homomorphisms $f_H\colon M(H)\to M'(H)$ for all $H\leq G$, commuting with all restriction, induction and conjugation maps.

A \emph{Green functor} is a Mackey functor $R$ for which each $R(H)$ is an (associative and unital) ring, restriction and conjugation maps are ring maps, and which satisfies the two \emph{Frobenius relations}
\[
\ind_K^H(\res_K^H(r) \cdot r') = r\cdot \ind_K^H( r'), \quad
\ind_K^H(r \cdot \res_K^H(r') ) = \ind_K^H( r) \cdot r' .
\]
A (left) \emph{Mackey module} over the Green functor $R$ is a Mackey functor $M$ such that $M(H)$ is endowed with the structure of a (left) $R(H)$-module ($H\leq G$), whose structure maps satisfy the analogs of the above Frobenius relations, and whose restriction and conjugation maps commute with the ring actions in the evident way (see \cite[\S1.1 and \S2.1]{Bouc97} for all details on the axioms). 

Mackey $R$-modules form an abelian category 
\[
R\MMackey := R\MMackey_\mathbb Z(G)
\]
where a morphism $f\colon M\to M'$ is a morphism of the underlying Mackey functors such that each $f_H\colon M(H)\to M'(H)$ is an $R(H)$-linear map.

\begin{Exa}
The first example of a Green functor is the \emph{Burnside ring} Green functor, which in this article we denote by~$B$. 
It is commutative and acts uniquely on all Mackey functors, so that $B\MMackey = \Mackey$.
See \cite[\S2.4]{Bouc97}.
\end{Exa}

\begin{Rec} \label{Rec:Dress_pic}
Besides the above \emph{subgroups picture} of Mackey functors, the alternative \emph{$G$-set (or Dress) picture} is often very useful. 
According to the latter, a Mackey functor is a pair $M=(M^*,M_*)$ of functors $M^*\colon G\sset^\op\to \Ab$ and $M\colon G\sset\to \Ab$ (a ``bifunctor'') agreeing on objects, $M^*(X)=M_*(X)=:M(X)$ for all finite $G$-sets~$X$, and which send disjoint unions to direct sums and satisfy a base-change condition for pullbacks. 
A morphism $M\to M'$ is a transformation $\{f_X\colon M(X)\to M'(X)\}_X$ which is natural for both functorialities.
The equivalence of the two definitions takes a bifunctor $(M^*,M_*)$ and yields a Mackey functor for the subgroups picture by setting $M(H):= M(G/H)$ for all subgroups $H\leq G$ as well as $\res^H_K:=M^*(G/K\twoheadrightarrow G/H)$, $\ind^G_K:=M_*(G/K\twoheadrightarrow G/H)$, and $c_{g,H}:=M^*(G/{{}^gH}\overset{\sim}{\to} G/H)$ (here $G/K\twoheadrightarrow G/H$ is the quotient map and $G/{}^gH\overset{\sim}{\to} G/H$ is the isomorphism $\gamma {}^gH\to \gamma gH$).  
Conversely, a subgroups-picture Mackey functor extends essentially uniquely to a bifunctor $(M^*,M_*)$ via the same formulas, by decomposing $G$-sets into orbits. See \cite[\S1.1.2]{Bouc97}.
\end{Rec}

\begin{Rec}
\label{Rec:tensor-Mackey}
If the Green functor $R$ is commutative, $R\MMackey$ admits a symmetric monoidal structure. 
We briefly recall its definition from \cite[\S6.6]{Bouc97} in terms of the $G$-set picture (\Cref{Rec:Dress_pic}). 
Given two $R$-modules $M,N$, the value at $X\in G\sset$ of their tensor product $M\otimes_RN$ is 
\[
\big( M\otimes_R N\big)(X) := \left( \bigoplus_{\theta:Y\to X} M(Y) \otimes_\mathbb Z N(Y) \right) / \mathcal J,
\]
where the sum is taken over all maps $\theta$ of finite $G$-sets, and $\mathcal J$ is the subgroup generated by the following three sets of relations:
\begin{align}
 M_*(f)(m)\otimes n' - m \otimes &N^*(f)(n') \label{eq:rel1and2}, 
\quad  M^*(f)(m')\otimes n - m'\otimes M_*(f)(m'), \\
& (r\cdot m) \otimes n - m\otimes (r\cdot n) \label{eq:rel3}
\end{align}
for all $m\in M(Y)$, $n\in N(Y)$, $m'\in M(Y')$, $n'\in N(Y')$, $r\in R(Y)$, and all $G$-set maps $f\colon Y\to Y'$ with $\theta' f = \theta$. 
Here  $M^*$ and $M_*$ are the two functorialities of~$M$, extending the restriction \& conjugation maps and the induction maps of the subgroup picture, respectively.

If we only quotient out the relations \eqref{eq:rel1and2}, we get the value at $X$ of the tensor product $M\otimes_B N$ (denoted $M\hat\otimes N$ in~\cite{Bouc97}) of the underlying plain Mackey functors, \ie the special case when $R$ is the Burnside ring~$B$. 
The structure maps of the Mackey functor $M\otimes_R N$ are inherited via the quotient map $M\otimes_B N \to M\otimes_R N$ from those of $M\otimes_B N$, as described in \cite[Prop.\,1.6.2]{Bouc97}.

The tensor unit is $R$, and the unitality, associativity and symmetry isomorphisms of the tensor structure are inherited from those of $\otimes_\mathbb Z$ in the obvious way.
\end{Rec}

\section{Brauer quotients of Mackey modules}
\label{sec:GreenlessMay-Module-Mackey}%
\medskip
In order to precisely state \Cref{Thm:intro-splitting}, we still need to define Brauer quotients and the rationalization of a Green functor.
Slightly more generally:

\begin{Not} \label{Not:rationalization}
For any subring $Q\subseteq \mathbb Q$ and any Green functor~$R$, we write $R_Q$ for the $Q$-linearization of~$R$, that is the Green functor defined by $R_Q(H):=Q\otimes_\mathbb Z R(H)$ for all $H\leq G$, and similarly for the structural maps. 
The module category $R_Q\MMackey$ can be identified with the full subcategory of $R\MMackey$ containing all \emph{$Q$-local} (or \emph{$Q$-linear}) Mackey modules, \ie those taking values in $Q$-modules.
(If necessary, recall that every subring $Q\subseteq \mathbb Q$ is a localization of~$\mathbb Z$.)
The \emph{rationalization} of $R$ is the Green functor $R_\mathbb Q$, the case when $Q= \mathbb Q$.
\end{Not}

\begin{Def}[{\cite[\S54]{Thevenaz95}}]  \label{Def:Brauer-quot}
For any Mackey functor~$M$ for $G$ and subgroup $H\leq G$, the \emph{Brauer quotient} of $M$ at $H$ is the quotient of abelian groups 
\[
\overline{M}(H) := M(H) / \left( \sum_{H' \lneq H} \ind_{H'}^H (M(H'))\right)
\] 
obtained by killing all elements induced from proper subgroups of~$H$.
\end{Def}

\begin{Rems}  \label{Rems:Psi}
Let $M$ be a Mackey functor and $H\leq G$ a subgroup. In a series of successive easy observations, we will turn the Brauer quotient construction into a suitable functor~$\Br_H$. 
\begin{enumerate} 
\item 
For $g\in N_G(H)$, the conjugation map $c_g:=c_{g,H}\colon M(H)\to M({}^gH)$ is an automorphism of~$M(H)=M({}^gH)$. 
By an axiom of Mackey functors, $c_g = \id$ whenever $g\in H$. 
We thus obtain an action of $W_G(H)$ on $M(H)$. 
By another axiom of Mackey functors, namely the commutation relations 
\[ c_g \ind^H_{H'} = \ind^{{}^gH}_{{}^gH'} c_g  \quad (\textrm{for } g\in G, H'\leq H\leq G),
\]
this $W_G(H)$-action descends on the Brauer quotient $\overline{M}(H)$. 
We call this action (both on $M(H)$ and $\overline{M}(H)$) the \emph{conjugation action} and still denote it by~$c$.
 \item 
 For a Green functor $R$, the subgroup $\sum_{H'<H}\ind_{H'}^{H}(R(H))\subseteq R(H)$ is an ideal thanks to the (left) Frobenius relation. 
 Therefore $\overline{R}(H)$ is a ring.
\item 
By definition, if $M= R$ is a Green functor the conjugation action on $R(H)$ is by ring automorphisms. This descends to $\overline{R}(H)$ by part~(b), hence we may construct the skew group ring  
 \[ 
 \overline{R}(H)\rtimes_c W_G(H) 
 \]
 as in~\S\ref{sec:skew-group-rings}.
 \item 
 If $M$ is a Mackey module over the Green functor~$R$, the Brauer quotient $\overline{M}(H)$ becomes a left $\overline{R}(H)$-module. 
 This follows similarly to (b) and using the left Frobenius axiom for Mackey modules.
\item 
Altogether, when $M$ is a Mackey module over~$R$, the conjugation action of~(a) and the $\overline{R}(H)$-action of~(d) turn $\overline{M}(H)$ into a left $\overline{R}(H)\rtimes_c W_G(H)$-module. 
Indeed, this follows immediately from the Mackey module axiom
\[
c_g(r \cdot m) = c_g(r) \cdot c_g(m)
\]
(for $g\in G$, $r\in R(H)$, $m\in M(H)$) together with \Cref{Rem:skew-modules}.
\item Suppose $f\colon M\to M'$ is a morphism of Mackey $R$-modules. 
Since the components of $f$ commute with the induction maps  of $M$ and~$M'$, $f_H$ descends to a homomorphism $\overline{f_H}\colon \overline{M}(H)\to \overline{M'}(H)$, which is $\overline{R}(H)$-linear because $f_H$ is $R(H)$-linear by definition. 
Since the components of $f$ commute with the conjugation maps, $\overline{f_H}$ commutes with the conjugation actions. In other words, $\overline{f_H}$ is a morphism of $\overline{R}(H)\rtimes_c W_G(H)$-modules.

\end{enumerate}
\end{Rems}

\begin{Def} \label{Def:Psi_H}
Write
\[
\Br_H \colon R\MMackey \longrightarrow  \overline{R}(H)\rtimes_c W_G(H) \MMod
\]
for the functor  resulting from~\Cref{Rems:Psi}, defined by $M\mapsto \Br_H(M):= (\overline{M}(H), c)$ on objects and $f\mapsto \Br_H(f):=\overline{f_H}$ on morphisms. We can now state:
\end{Def}

\begin{Thm}\label{thm.RightAdjointMackMod}
	For any finite group~$G$ and any Green functor $R$ for~$G$, the functor
	\[
	\Br := (\Br_H)_{\mathrm{Cl}(H)} : R\MMackey\longrightarrow \underset{\mathrm{Cl}(H),\ H\leq G}{\prod} \overline{R}(H) \rtimes_c  W_G(H)\MMod
	\]
	%
	admits a right adjoint, which we denote by $\Phi$. 
	These two functors induce an adjoint equivalence 
	\[
	R_ Q \MMackey\simeq \underset{\mathrm{Cl}(H),\ H\leq G}{\prod}  Q\otimes_\mathbb Z \overline{R}(H) \rtimes_c  W_G(H)\MMod
	\]
	for any subring $Q\subseteq \mathbb Q$ containing $|G|^{-1}$ (see \Cref{Not:rationalization}).
\end{Thm}

\begin{Rem} \label{Rem:graded-cble-variants}
The theorem admits many easy variations, for which the exact same proof will go through smoothly. For instance, we may allow Mackey functors with values in a general commutative ring~$\mathbb K$ (\cf \Cref{Rem:intro-coefficients}), possibly graded. 
For our later purposes, we can take $\mathbb Z/2$-graded modules (over an ungraded~$R$) and restrict to countable abelian groups, on both sides, in order to obtain an equivalence
\[
R_ \mathbb Q \MMackey^{\mathbb Z/2}_{\aleph_1} \simeq \underset{\mathrm{Cl}(H),\ H\leq G}{\prod}  \mathbb Q\otimes_\mathbb Z \overline{R}(H) \rtimes_c  W_G(H)\MMod^{\mathbb Z/2}_{\aleph_1}
\]
between $\mathbb Z/2$-graded countable rational Mackey $R_\mathbb Q$-modules and the corresponding product category of $\mathbb Z/2$-graded countable modules. (The sense of the notation used here is that ``countable $=$ $\aleph_1$-small''.)
\end{Rem}

\begin{proof}[Proof of \Cref{thm.RightAdjointMackMod}]
To begin with, let us fix a subgroup $H\leq G$.

The existence of an adjunction $\Br_H\dashv \Phi_H$ is established in \cite[Lemma~11.6.1]{Bouc97}, in terms of the $G$-set picture (\Cref{Rec:Dress_pic}).
Beware of the notations: in \emph{loc.\,cit.}, the Green functor $R$ is denoted by~$A$; our functor $\Br_H$ is simply written $M\mapsto \overline{M}(H)$; its right adjoint $\Phi_H$ by $V\mapsto F\!P^G_{H,V}$; and the skew group ring $\overline{R}(H)\rtimes_cW_G(H)$ by $\overline{A}(H)\otimes \overline{N}_G(H)$.

Before we can prove the remaining claims, we need to identify the unit and counit of this adjunction. 
Explicitly, for every left $\overline{R}(H)\rtimes_cW_G(H)$-module~$V$, the Mackey $R$-module $\Phi_H(V)$ sends a finite $G$-set $X$ to 
\[
\map^{W_G(H)}( X^H , V),
\]
the set of $W_G(H)$-equivariant maps from the fixed-point set $X^H$ (equipped with the induced action of the Weyl group) into~$V$
 (see \cite[Def.\ on~p.\,296]{Bouc97} for its functoriality and $R$-action).
We immediately deduce from the proof in \emph{loc.\,cit.} that the bijection of Hom-sets for this adjunction,
\[
R\MMackey \big(M, \Phi_H (V) \big)\cong \overline{R}(H)\rtimes_c W_G(H) \MMod \big(\Br_H (M) , V\big ),
\]
sends a morphism $f \colon M\to \Phi_H(V)$ of left Mackey $R$-modules to the composite map
\[
\xymatrix{
\Br_H(M) \ar[r]^-{\Br_H(f)} &
 \Br_H( \Phi_H(V)) \ar[r]_-{\sim}^-{\varepsilon_{H,M}} & 
  V .
}
\]
Here $\varepsilon_{H,M}$ is the bijection 
\begin{align*}
\varepsilon_{H,V}\colon\Br_H\Phi_H(V)\cong \overline{(\Phi_HV)}(H) = \map^{W_G(H)}\big( (G/H)^H, V \big) & \overset{\sim}{\longrightarrow} V \\
 f &\longmapsto f(eH) ,
\end{align*}
where the first identification is simply $\overline{(\Phi_HV)}(H)= (\Phi_HV)(H)$ (indeed: for each subgroup $K\leq H$, a coset $gK\in G/K$ is an $H$-fixed point iff $H^g \subseteq K$, hence if $K\lneq H$ is proper we must have $(\Phi_HV)(K) = 0$), and the second one is due to $(G/H)^H$ being a transitive $W_G(H)$-set. 
One verifies easily that $\varepsilon_{H,V}$ is $W_G(H)$-equivariant, $\overline{R}(H)$-linear and natural in~$V$, from which it follows that the natural transformation $\varepsilon_H\colon \Br_H\Phi_H\to \Id$ must be the counit of the adjunction $\Br_H \dashv \Phi_H$. Note in particular, for later use, that the counit is an isomorphism. 

The unit $\eta_H$ is similarly determined by going backward in the bijection; its component at the object $M$ is the natural morphism $\eta_{H,M}\colon M\to \Phi_H \Br_H(M)$ of Mackey $R$-modules whose component at $X\in G\sset$ is given by 
\begin{align*}
\eta_{H,M}(X) \colon M(X) & \longrightarrow \map^{W_G(H)}(X^H, \overline{M}(H))  \\
m & \longmapsto \left(\;\;  x \mapsto [ M^*(\ev_x)(m) ] \;\;\right) ,
\end{align*}
where $\ev_x\colon G/H\to X$ is the map  $\gamma H\mapsto \gamma x$ of $G$-sets (which is well-defined since $x\in X^H$), and square brackets denote the equivalence class in the Brauer quotient.
Note in particular that the component at an orbit $X= G/K$ takes the form
\begin{align*}
\eta_{H,M}(G/K) \colon M(G/K) & \longrightarrow \map^{W_G(H)}((G/K)^H, \overline{M}(H))  \\
m & \longmapsto \left(\;\;  gK \mapsto [ c_g \res^K_{H^g}(m) ] \;\;\right) 
\end{align*}
(indeed $M^*(\ev_x) = c_g \Res^K_{H^g}$ when $x= gK$). 
For the subgroup $K=G$, writing $M(G)=M(G/G)$ as in the subgroups picture, this simplifies to the map:
\begin{equation} \label{eq:Brauer_hom}
\eta_{H,M}(G)\colon M(G) \longrightarrow \overline{M}(H)^{W_G(H)} , \qquad m\mapsto [\Res^G_H(m)].
\end{equation}

It is now time to sum the functors $\Phi_H$ in order to get the required right adjoint 
\begin{align*}
\Phi\colon \prod_{\mathrm{Cl}(H), H\leq G} \overline{R}(H)\rtimes_c W_G(H)  \longrightarrow R\MMackey, 
\quad\quad
(V_H )_H  \longmapsto \bigoplus_H \Phi_H(V_H)
\end{align*}
of the functor $\Br$ in \Cref{thm.RightAdjointMackMod}. 
The unit $\eta$ and counit $\varepsilon$ are the evident combinations of the units $\eta_H$ and counits $\varepsilon_H$ identified above. 

In order to conclude the proof of the theorem, we can now proceed in the same way as~\cite[Theorem 3.4.22]{Schwede18}, which is precisely the special case of the theorem where the Green functor $R$ is the Burnside ring Green functor~$B$, for which $B\MMackey = \Mackey$. 
Note that by forgetting $R$- and $\overline{R}(H)$-actions throughout in the above (and by working with the subgroups picture), we get exactly the same adjunction as in Schwede's proof.
We recall the argument for completeness.

Note that the functors, and therefore the adjunction $(\Br, \Phi, \eta, \varepsilon)$, can be restricted on both sides to the subcategories of $Q$-local objects. For the remainder of the proof we will only consider this restricted adjunction.
It is proved in \emph{loc.\,cit.\ }that, on $Q$-local modules, the unit of adjunction $\eta$ becomes invertible as a map of the underlying Mackey functors, hence also as a morphism of Mackey $R$-modules. 
Indeed, thanks to \eqref{eq:Brauer_hom}, we can identify the $G$-component of the unit $\eta_M$ with Th\'evenaz's Brauer homomorphism $M(G)\to ( \prod_{H} \overline{M}(H))^G\cong \prod_{\mathrm{Cl}(H)} \overline{M}(H)^{W_G(H)} $, which is known to become bijective after inverting~$|G|$; see \cite{Thevenaz88} or \cite[Prop.\,3.4.18]{Schwede18}.
Then one can easily deduce from this that the $H$-component for each $H\leq G$ is also invertible, see \emph{loc.\,cit.}

Since (always between $Q$-local objects) the unit of adjunction is invertible, it follows formally that the left adjoint $\Br$ is fully faithful, and it remains to prove it is essentially surjective.
To see this, take an arbitrary family $(V_H)_H$ of $Q$-local modules over the skew rings.
By the additivity of the functor~$\Br$, it suffices to show that its essential image contains each summand~$V_H$.
Let $e_H= (e_{H,K})_K$ be the idempotent endomorphism of the object
\[ \Br\Phi (V_H)\;\;\in\;\; \prod_{\mathrm{Cl}(K)} Q\otimes_\mathbb Z \overline{R}(K)\rtimes_c W_G(K) \MMod 
\]
whose components are $e_{H,H}:= \id_{V_H}$ and $e_{H,K}:= 0 $ for $K\neq H$.
Note that $\Img(e_H) = \Br_H\Phi_H(V_H)$ by construction, and that the latter module is isomorphic to $V_H$ via the counit $\varepsilon_{H,V_H}$ of the adjunction $\Br_H\dashv \Phi_H$.

Since $\Br$ is fully faithful on $Q$-local objects, there exists an idempotent $\tilde e_{H}$ on $\Phi(V_H)\in R_Q\MMackey$ with $\Br(\tilde e_{H})= e_{H}$. Hence $\Phi_H(V_H) \cong \Img(\tilde e_{H})\oplus \Img(\id - \tilde e_H)$, and by additivity we deduce that 
$ \Br(\Img \tilde e_H) \cong \Img (\Br  \tilde e_H) = \Img (e_H) \cong V_H$. 
Therefore $V_H$ belongs to the essential image of~$\Br$.

This completes the proof of \Cref{thm.RightAdjointMackMod}.
\end{proof}

\begin{Rem}
In order to prove \Cref{thm.RightAdjointMackMod} for a general Green functor~$R$, it would suffice to verify that the adjunction $(\Br,\Phi,\eta,\varepsilon)$ for the case $R=B$ of the Burnside ring, as made explicit in \cite[Theorem 3.4.22]{Schwede18}, lifts to the appropriate module categories.
However, if we want to work exclusively with the subgroups picture of Mackey functors as Schwede does, the definition of the $R$-action on $\Phi_H(V)$ and the resulting verifications that units and counits are compatible (which we have also carried out in details!) are surprisingly intricate and amount to several pages of calculations. The above approach, which partly uses general constructions within the $G$-set picture, is much more economical. 
It also reveals a more conceptual definition of $\Br_H$ as the composite of certain canonical adjoints (\cf \Cref{sec:tensor}).
\end{Rem}

\begin{Rem} \label{Rem:idempotents}
The rational Burnside ring of~$G$, $B(G)_\mathbb Q = \mathbb Q \otimes_\mathbb Z B(G)$, canonically acts on rational Mackey functors, hence its idempotent elements split their category $\Mackey_\mathbb Q$.  
These idempotents $e_H\in B(G)_\mathbb Q$ are well-known to be parametrized by the subgroups $H$ of~$G$, and we obtain a complete orthogonal family by choosing one for each $G$-conjugacy class~$\mathrm{Cl}(H)$.
The resulting splitting is in fact the same as appears on the right-hand side in~\eqref{eq:intro-equiv}; see the remarks before \cite[Theorem~3.4.22]{Schwede18}.
\end{Rem}

\section{Comparing the tensor structures}
\label{sec:tensor}%
\medskip
We now consider the tensor structures on our categories.

\begin{Thm}
\label{Thm:Brauer_tensor}
Suppose $R$ is a commutative Green functor for~$G$. 
Then for each $H\leq G$ the Brauer quotient functor $\Br_H$ of \Cref{Def:Psi_H} is strong symmetric monoidal with respect to the tensor structures of  \Cref{Rec:tensor-Mackey} and~\Cref{Rem:tensor-skew}. 
Therefore the functor $\Br$ of \Cref{thm.RightAdjointMackMod} is symmetric monoidal (where its target category is equipped with the product tensor structure), and it induces an equivalence
\[
R_Q \MMackey \simeq
{\underset{\mathrm{Cl}(H),\ H\leq G}{\prod}  Q\otimes_\mathbb Z \overline{R}(H) \rtimes_c  W_G(H)\MMod }
\]
of tensor categories for each subring $Q\subseteq \mathbb Q$ containing~$|G|^{-1}$.
\end{Thm}

\begin{proof}
As with \Cref{thm.RightAdjointMackMod}, we are essentially going to reduce the proof to the case of the Burnside Green functor~$B$, for which $B\MMackey = \Mackey$.
As mentioned in the introduction, in this case the result is known to be true by experts but we nonetheless provide a (sketch) of proof for completeness.

We need to find a symmetric monoidal structure on~$\Br_H$, that is an isomorphism
\[
\unit \overset{\sim}{\longrightarrow} \Br_H(\unit)
\]
and a natural isomorphism (for $M,N\in R\MMackey$)
\begin{equation} \label{eq:multipl_Psi}
\Br_H(M) \otimes \Br_H(N)
\overset{\sim}{\longrightarrow}
 \Br_H( M \otimes_R N) 
\end{equation}
satisfying unitality and associativity constraints.
For the first isomorphism, we simply take the identity map of the Brauer quotient ring $\overline{R}(H)$.
For the second one, consider the following commutative diagram of abelian groups:
\begin{equation} \label{eq:big_diagram}
\vcenter{
\xymatrix@R=10pt{
 M(H) \otimes_\mathbb Z N(H)
  \ar@{->>}[dd]_{br \,\otimes\, br} 
  \ar[r] & 
 { \underset{\theta\colon Y\to G/H}{\bigoplus} M(Y) \otimes_\mathbb Z N(Y) }
     \ar@{->>}[dd]
      \ar@{->>}[dr]^-{\textrm{\eqref{eq:rel1and2}}} & \\
     && (M\otimes_B N)(H) 
      \ar@{->>}[dl]_{br}
       \ar@{->>}[dd]^{\textrm{\eqref{eq:rel3}}}  \\
\overline{M}(H) \otimes_\mathbb Z \overline{N}(H)
 \ar@{-->}[r]^-{\sim} 
      \ar@{->>}[dd] &
 \overline{(M\otimes_B N)}(H) 
       \ar@{->>}[dd] & \\
       &&  (M \otimes_R N)(H)
        \ar@{->>}[dl]_{br} \\
\overline{M}(H) \otimes_{\overline{R}(H)} \overline{N}(H)
 \ar@{-->}[r]^-{\sim} &
 \overline{(M\otimes_RN)}(H) &
}
}
\end{equation}
The direct sum is taken over all maps $\theta\colon Y\to G/H$ of finite $G$-sets (as in \Cref{Rec:tensor-Mackey}, which uses the $G$-sets picture and in particular $M(G/H)= M(H)$). 
The top horizontal arrow is the inclusion of the summand indexed by $\theta = \id_{G/H}$, and all arrows $\twoheadrightarrow$ denote quotient maps, with in particular $br$ denoting a Brauer quotient $P(H)\twoheadrightarrow \overline{P}(H)$ at $H$ for the relevant Mackey functor~$P$.

We claim that the top inclusion map descends to a bijection between the bottom quotients, and that this yields the required multiplication map~\eqref{eq:multipl_Psi}.
Indeed, suppose for a moment that it does descend to a well-defined bijection. 
It is immediate to see that this map commutes with the diagonal induced conjugation actions of $W_G(H)$ (\cf \cite[Prop.\,1.6.2]{Bouc97} for the conjugation action on the right-hand side) and with the induced $\overline{R}(H)$-actions, so it is a morphism of $\overline{R}(H)\rtimes_c W_G(H)$-modules.
It is equally easy to see that this map is natural in $M,N$ and that it satisfies the unit and associativity axioms of a monoidal structure on~$\Br_H$, as these properties are inherited from the analog properties of the inclusion map at $\theta = \id_{G/H}$.

So it suffices to show the induced map exists and is bijective. Moreover, if we knew the claim holds for $R=B$ the Burnside ring---that is, if we knew there is a well-defined horizontal bijection in the middle of~\eqref{eq:big_diagram}---then we could conclude for general $R$ as well. 
This is because then, if we quotient out the relations \eqref{eq:rel3} on $(M\otimes_BN)(H)$ to obtain $(M\otimes_RN)(H)$, the effect on $\overline{M}(H)\otimes_\mathbb Z \overline{N}(H)$ is precisely that of killing the elements $rm\otimes n - m\otimes rn$, as wished.

All in all, it only remains to show that there is a well-defined induced bijection in the middle, \ie for the tensor product of plain Mackey functors.
It would be nice to verify this directly from the relations \eqref{eq:rel1and2} and \Cref{Def:Brauer-quot}, but this appears to be difficult. 
Another approach would be to use the action of the rational Burnside ring (\cf \Cref{Rem:idempotents}) and the properties of its idempotent elements, which indeed shows the required bijection but only rationally; see \cite[\S4.8]{BarnesKedziorek22} for this approach.
We therefore conclude by sketching a proof which works integrally; it exploits the full power of the $G$-set picture of Mackey functors.

The Brauer quotient at~$H$ admits a different description in terms of the general theory of \cite[Ch.\,8-10]{Bouc97} for constructing functors between categories of Mackey functors.
Namely, let $U = {}_{W_G(H)}U_G$ be the set $H\backslash G$ of left cosets equipped with the right $G$-action and the induced left $W_G(H)$-action  by multiplication.
This biset yields a functor $-\circ U\colon \Mackey(W_G(H)) \to \Mackey(G)$ by $(M\circ U )(X)= M(U\circ_{W_G(H)} X)$ where $\circ_{W_G(H)}$ is the biset composition of \cite[\S8.1]{Bouc97}.
The functor $-\circ U$ has a left adjoint, denoted $\mathcal L_U$.
The biset $U$ is evidently the following composite of bisets
\[
U = {}_{W_G(H)}(H\backslash G)_G \cong \underbrace{{}_{W_G(H)} W_G(H)_{N_G(H)}}_{U_2} \;\circ\; \underbrace{ {}_{N_G(H)} G_G}_{U_1} \;\;,
\]
for which $-\circ U_1 = \Infl_{W_G(H)}^{N_G(H)}$ and $-\circ U_2= \Ind^G_{N_G(H)}$ are, respectively, the inflation and induction functors for Mackey functors. 
Hence by taking left adjoints we get
\[
\mathcal L_U \cong \mathcal L_{U_1}\circ \mathcal L_{U_2} = (-)^H \circ \Res^G_{N_G(H)}
\]
where $(-)^H\colon \Mackey(N_G(H))\to \Mackey(W_G(H))$ is the functor introduced in \cite[Prop.\,9.9.1]{Bouc97}. 
As remarked just after \emph{loc.\,cit.}, $M^H$ coincides with the functor $(-)^+$ of~\cite{ThevenazWebb95}, for which the evaluation at the terminal $W_G(H)$-set $1= W_G(H)/W_G(H)$ is precisely the Brauer quotient at~$H$ (by direct inspection; \cf \cite[p.\,1872]{ThevenazWebb95}).
Note also that the Brauer quotient of $M\in \Mackey(G)$ at $H$ is the same as the Brauer quotient of $\Res^G_{N_GH}(M)\in \Mackey(N_G(H))$ at~$H$ (again by direct inspection, using that $\Res^G_{N_G(H)}(M)= M(\Ind^G_{N_G(H)} - )$).
In sum, $\Br_H$ is isomorphic to the composite functor 
\[
\xymatrix{
\Mackey(G) 
 \ar[r]^-{\mathcal L_U} &
\Mackey(W_G(H)) 
 \ar[r]^-{\ev_1} &
 W_G(H)\MMod .
}
\]
Now, there is a natural isomorphism 
\begin{equation} \label{eq:mult_LU}
\mathcal L_U (M) \otimes_B \mathcal L_U(N) 
\cong \mathcal L_U( M \otimes_B N)
\end{equation}
 by \cite[Prop.\,10.1.2]{Bouc97} and because the right $G$-action on $U$ is transitive.
On the other hand, since $1$ is terminal, we easily see from the relations \eqref{eq:rel1and2} that 
\begin{equation} \label{eq:mult_ev1}
\ev_1 (M') \otimes_\mathbb Z \ev_1 (N')\cong \ev_1(M' \otimes_B N').
\end{equation}
The isomorphism \eqref{eq:multipl_Psi} is then obtained by combining \eqref{eq:mult_LU} and~\eqref{eq:mult_ev1}. 
More precisely, by chasing through the various adjunctions and the construction of \eqref{eq:mult_LU} (going back  all the way to the diagonal maps $\delta^U_{X,Y}\colon U\circ (X\times Y)\to (U\circ X)\times (U\circ Y)$ of \cite[\S10.1]{Bouc97}), one can verify that the combined isomorphism is simply the map induced from the inclusion at $\theta=\Id_{G/H}$, as claimed above. 
\end{proof}

\section{Application to equivariant KK-theory}
\label{sec:Module-Mackey-Gcell}%
\medskip

Let $\KK^G$ denote the $G$-equivariant Kasparov category of separable complex $G$-C*-algebras. 
As proved in~\cite{MeyerNest06}, this is a tensor triangulated category with arbitrary countable coproducts.
Following \cite{DellAmbrogio14},  we denote by $\Cell{G}\subseteq \KK^G$ the full subcategory of \emph{$G$-cell algebras}, the localizing subcategory  generated by the function $G$-C*-algebras $\Cont(G/H)$ for all $H\leq G$. 
We also write
\[
 \Cell{G}_\mathbb Q 
 \]
for the category of \emph{rationalized} $G$-cell algebras. 
More precisely, $\Cell{G}_\mathbb Q$ is defined to be the quotient $\Cell{G}/\Loc{\cone(f)\mid f\in \mathbb Z^\times}$, \ie the Bousfield localization at the central multiplicative system of maps $\mathbb Z^\times \subseteq \Rep(G)= \End_{\Cell{G}}(\unit)$.
It inherits nice generating properties  from $\Cell{G}$: it is a rigidly-compactly${}_{\aleph_1}$ generated tt-category, characterized by the property that $\Cell{G}_\mathbb Q(A, B)=\mathbb Q \otimes_\mathbb Z \KK^G_*(A, B)$ whenever $A$ is a compact${}_{\aleph_1}$ object of $\Cell{G}$ and $B$ is arbitrary; see \cite[\S2]{DellAmbrogio10}.

In this final section, we apply our previous results to provide an explicit algebraic model of $\Cell{G}_\mathbb Q$ in the form of \Cref{Thm:rational_cell_vs_Mack}.

Recall that the complex representation rings $\Rep(H)$ of all subgroups $H\leq G$ organize themselves into a Green functor (for~$G$) which we denote $\Rep^G$.
This is one of the most classical and ubiquitous examples of Green functors, as is its rationalization~$\Rep^G_\mathbb Q$ (\Cref{Not:rationalization}). It is well-known how its Brauer quotients look like:

\begin{Prop} \label{prop.RationalRepThev}
\label{Prop:Brauer-for-RepQ}
Let $\overline{\Rep^G_\mathbb Q}(H)$ be the Brauer quotient  (\Cref{Def:Brauer-quot}) of $\Rep^G_\mathbb Q$ at a subgroup $H\leq G$. Then:
	\begin{enumerate}
		\item There is an isomorphism of rings
			\begin{equation*}
				\overline{\Rep^G}(H)\cong \left\{\begin{array}{cl}
									\mathbb{Q}(\zeta_{n})&\text{if $H=\langle g\rangle$ is cyclic of order~$n$}\\\
									0&\text{otherwise},
					   			  \end{array}\right.
			\end{equation*}
			where $\mathbb{Q}(\zeta_{n})$ is the field extension of $\mathbb Q$ by a primitive $n$-th root of unity~$\zeta_{n}$. 
				\item Under (a), the conjugation action $c$ of $W_G(H)$ on $\overline{\Rep^G}(H)$ is through the Galois group. More precisely, fixing the generator $g$ of the cyclic subgroup~$H$ determines a group homomorphism $m_H\colon W_G(H)\to (\mathbb Z/n)^\times$ through the formula $ g^{m_H(w)}= w^{-1}gw$ for all $w\in W_G(H)$.
		Then $c_w (\zeta_{n})= \zeta_{n}^{m_H(w)}$.
	
	\end{enumerate}
\end{Prop}

\begin{proof} For part~(a) see \eg\ \cite[Section 9]{Thevenaz88}. 
For~(b) see \eg\ \cite[\S4]{BHIKM22}.
\end{proof}

We immediately deduce:

\begin{Thm}\label{thm.EquivRationalMackMod}
		For every finite group $G$, there is an equivalence
		\[ 
		\Rep^G_{\mathbb{Q}}\MMackey
		\overset{\sim}{\longrightarrow} 
		{\prod_{\mathrm{Cl}(H),\, H\leq G\; \textrm{cyclic}}}  \mathbb{Q}(\zeta_{|H|}) \rtimes_c W_G(H)\MMod, 
		\]
		of symmetric monoidal categories, where the right-hand side is the product of the tensor categories, as in \Cref{Rem:tensor-skew}, of modules over the skew group rings of \Cref{Prop:Brauer-for-RepQ}. 
		In particular, $\Rep^G_{\mathbb{Q}}\MMackey$ is a semisimple abelian tensor category.
\end{Thm}

\begin{proof}
Just combine \Cref{Lem:semisimple}, Theorems \ref{thm.RightAdjointMackMod} and~\ref{Thm:Brauer_tensor}, and \Cref{prop.RationalRepThev}.
\end{proof}

By one of the main results of \cite[\S4]{DellAmbrogio14}, we know that the usual equivariant K-theory groups $k^G(A)(H):= K^H_0(A)\oplus K^H_1(A)$ of each object $A \in \KK^G$ can be naturally upgraded to form a Mackey module $k^G(A)$ over the Green functor~$\Rep^G$, as in \Cref{sec:Mackey}, which moreover is $\mathbb Z/2$-graded and countable.
 (Alternatively, the same conclusion can also be deduced from the Green 2-functor structure of $G\mapsto \KK^G$ via Hom-decategorification; see \cite[Example~12.20]{DellAmbrogio22}.)
After restricting to $G$-cell algebras, this defines a functor
\[
k^G \colon \Cell{G} \hookrightarrow \KK^G \longrightarrow \Rep^G\MMackey^{\mathbb Z/2}_{\aleph_1}
\]
into the category of $\mathbb Z/2$-graded countable Mackey $\Rep^G$-modules.
As a homological functor into an abelian category, it can be used to do relative homological algebra in $\Cell{G}$ and set up spectral sequences. 
This is the content of~\cite{DellAmbrogio14}. 
By restricting further to rational $G$-cell algebras, we obtain a (simpler) homological functor
\begin{equation}  \label{eq:rational-K}
k^G_\mathbb Q \colon \Cell{G}_\mathbb Q  \longrightarrow \Rep^G_\mathbb Q\MMackey^{\mathbb Z/2}_{\aleph_1}
\end{equation}
into the category of $\mathbb Z/2$-graded countable Mackey modules over~$\Rep^G_\mathbb Q$, which is what will be used in the proof of the next result.

\begin{proof}[Proof of \Cref{Thm:rational_cell_vs_Mack}]
First notice that the abelian category on the right-hand side is semisimple by \Cref{Lem:semisimple}, since the coefficient ring in each skew group ring is a field, and because the graded countable modules in a semisimple category of modules again form a semisimple abelian category.

We claim that \eqref{eq:rational-K} is a tensor equivalence; by composing it with the tensor equivalence of \Cref{thm.EquivRationalMackMod} (in the  graded countable variant as in \Cref{Rem:graded-cble-variants}), we then get the desired result.

To prove the claim, we may for instance set up relative homological algebra as in \cite{DellAmbrogio14} except that we use the homological functor $k^G_\mathbb Q$ instead of~$k^G$.
It follows by rationalizing both sides of \cite[Thm.\,4.9]{DellAmbrogio14}, together with \cite[Prop.\,5.6]{DellAmbrogio14}, that \eqref{eq:rational-K} is the universal $\mathcal I$-exact stable homological functor for the homological ideal $\mathcal I := \ker(k^G_\mathbb Q)$ of morphisms in $\Cell{G}_\mathbb Q$. 
The associated UCT spectral sequence \cite[Thm.\,5.12]{DellAmbrogio14} collapses on the first page for all $A,B\in \Cell{G}_\mathbb Q$, which implies that \eqref{eq:rational-K} is a fully faithful functor.
It is also essentially surjective, as the objects $k^G_{\mathbb Q}(\Cont(G/H))$ for $H\leq G$ generate the whole category of Mackey modules through direct sums and retracts (again by semisimplicity, and because they are projective generators). 
Thus \eqref{eq:rational-K} is an equivalence. 

To see that is it a symmetric monoidal functor, note that the (rational analogue) of \cite[Lemma 5.15]{DellAmbrogio14} yields for every subgroup $H\leq G$ a canonical isomorphism
\[
k^G_\mathbb Q(\Cont(G/H)) \otimes_{\Rep^G_\mathbb Q} k^G_\mathbb Q(B)
\overset{\sim}{\to} 
k^G_\mathbb Q\big( \Cont(G/H)\otimes B \big) 
\]
in $\Rep^G_\mathbb Q\MMackey^{\mathbb Z/2}_{\aleph_1}$, natural in $B\in \Cell{G}_\mathbb Q$.
Using the semisimplicity of $\Cell{G}_\mathbb Q$, we can easily extend these to an isomorphism 
\[
k^G_\mathbb Q(A) \otimes_{\Rep^G_\mathbb Q} k^G_\mathbb Q(B)
\overset{\sim}{\to} 
k^G_\mathbb Q\big( A\otimes B \big) 
\]
natural in $A,B\in \Cell{G}_\mathbb Q$.
This provides the required monoidal structure on \eqref{eq:rational-K}, where the unit component 
$\Rep^G_\mathbb Q \overset{\sim}{\to} k^G_\mathbb Q(\mathbb C)$ is the usual identification of the representation ring with the equivariant K-theory of a point.
\end{proof}

\bibliographystyle{alpha}
\newcommand{\etalchar}[1]{$^{#1}$}

\printindex
\end{document}